\documentclass[11pt]{amsart}
\usepackage{amsmath,amssymb,amsthm}
\usepackage[margin=1in]{geometry}
\usepackage[hidelinks]{hyperref}
\usepackage{graphicx}
\usepackage{tikz}
\theoremstyle{plain}
\newtheorem{theorem}{Theorem}[section]

\newtheorem{lemma}[theorem]{Lemma}

\theoremstyle{definition}

\theoremstyle{remark}
\newtheorem{remark}[theorem]{Remark}

\usepackage{color}

\newcommand{\g}{\mathfrak{g}}
\newcommand{\h}{\mathfrak{h}}
\newcommand{\Phire}{\Phi_{\mathrm{re}}}
\newcommand{\Phiim}{\Phi_{\mathrm{im}}}
\newcommand{\Phiimp}{\Phi_{\mathrm{im}}^{+}}

\title[Finite root sums]{Finite Root Sums in Symmetrizable Kac--Moody Algebras}

\author[A. Douglas]{Andrew Douglas$^{1,2,3}$}
\address[]{$^1$Department of Mathematics, New York City College of Technology, City University of New York, Brooklyn, NY, USA.}
\email{adouglas@citytech.cuny.edu}
\address[]{$^2$Ph.D. Programs in Mathematics and Physics, CUNY Graduate Center, City University of New York, New York, NY, USA.}
\address[]{$^3$Department of Mathematics, University of Toronto, Toronto, ON, Canada}

\keywords{Symmetrizable Kac--Moody algebra, root system, closed subset,
finite-sum closure, wide subalgebra}
\subjclass[2020]{Primary 17B67; Secondary 17B22, 17B10}

\begin{document}

\begin{abstract}
Let $\g$ be a symmetrizable Kac--Moody algebra with root system $\Phi$. A subset $S\subseteq\Phi$ is \emph{closed} if $\alpha,\beta\in S$ and
$\alpha+\beta\in\Phi$ imply $\alpha+\beta\in S$. For finite and affine
Kac--Moody algebras, this closure property is known to extend from pairs
to arbitrary finite sums that are roots. We prove that every closed subset of the root system of any
symmetrizable Kac--Moody algebra is closed under arbitrary finite
sums that are roots.
The proof uses real-root strings, a simultaneous separation lemma for positive imaginary roots, and a lemma on when the sum of two positive imaginary roots is again a root. As an application, we extend a necessary condition
for wideness of regular subalgebras from finite and affine types to all symmetrizable
Kac--Moody algebras: the symmetric closure of the roots associated
to a wide regular subalgebra is the entire root system.
\end{abstract}

\maketitle

\section{Introduction}

Let $\g$ be a symmetrizable Kac--Moody algebra with Cartan subalgebra
$\h$ and root system $\Phi\subseteq\h^*$. We denote the induced symmetric
bilinear form on $\h^*$ by $(\cdot\mid\cdot)$. Fix a set of positive
roots $\Phi^+$ with simple roots $\Pi=\{\alpha_1,\dots,\alpha_\ell\}$.
The sets of real and imaginary roots are
\[
\Phire=\{\alpha\in\Phi:(\alpha\mid\alpha)>0\},
\qquad
\Phiim=\{\alpha\in\Phi:(\alpha\mid\alpha)\le0\},
\]
and we write $\Phiimp=\Phiim\cap\Phi^+$ for the set of positive
imaginary roots. A subset $S\subseteq\Phi$ is \emph{closed} if
$\alpha,\beta\in S$ and $\alpha+\beta\in\Phi$ imply $\alpha+\beta\in S$.

For finite type \cite[Lem.~3.2]{DouglasRepka} and affine type
\cite[Prop.~3.4]{DouglasAli}, this closure property was shown to extend
from pairs to arbitrary finite sums that are roots. We say that $\Phi$
has the \emph{finite-sum closure property}, or simply \emph{finite-sum
closure}, if for every closed subset $S\subseteq\Phi$,
\[
\beta_1,\dots,\beta_n\in S,\qquad
\eta:=\beta_1+\cdots+\beta_n\in\Phi
\quad\Longrightarrow\quad
\eta\in S.
\]

Our main theorem shows that every symmetrizable Kac--Moody algebra
has the finite-sum closure property.

\medskip
\noindent\textbf{Main theorem.}
\emph{Let $\g$ be a symmetrizable Kac--Moody algebra with root system
$\Phi$. Then $\Phi$ has the finite-sum closure property.}
\medskip

The proof proceeds by induction on the number of summands.
The cases involving real roots use real-root strings and
negative-pairing arguments. For imaginary roots, we use Marquis's
simultaneous separation lemma \cite[Lem.~3.4]{Marquis}. Suppose two positive imaginary roots have connected union of supports
and one of them pairs nonpositively with every simple root.
Lemma~\ref{lem:connected-addition} shows that their sum is again
a positive imaginary root. Together, these two lemmas
let us handle imaginary summands of both signs.

As an application, we extend a necessary condition for wideness of
regular subalgebras from finite and affine types
\cite{DouglasRepka,DouglasAli} to all symmetrizable Kac--Moody algebras. We show that if $\mathfrak{s}$ is a wide regular subalgebra and $T$ is its
associated subset of roots, then
\[
[T\cup(-T)]=\Phi,
\]
where brackets denote closure. 

The article is organized as follows. Section~\ref{prelim}
collects the required root-theoretic facts and proves the
addition lemma for fundamental imaginary roots.
Section~\ref{main} proves the finite-sum closure property.
Section~\ref{sec:wide} gives the application to wide regular
subalgebras.

\section{Root-theoretic preliminaries}\label{prelim}

We write $Q=\bigoplus_{i=1}^{\ell}\mathbb Z\alpha_i$ for the root
lattice, $Q^+=\bigoplus_{i=1}^{\ell}\mathbb Z_{\ge0}\alpha_i$, and $W$
for the Weyl group. For $\gamma=\sum_i c_i\alpha_i\in Q$, its support
$\operatorname{supp}(\gamma)$ is the subdiagram induced by the vertices
$\alpha_i$ with $c_i\ne0$. Unions of supports are understood as
the subdiagrams induced by the unions of their vertex sets.
Every root has connected support, and $W$
preserves $\Phi$, $\Phi_{\mathrm{im}}^+$, and
$\Phi_{\mathrm{im}}^-$ \cite[\S\S2.3--2.4]{Marquis}.

For each simple root $\alpha_i$, let $\alpha_i^\vee\in\h$
be the corresponding simple coroot. Following \cite[\S2.4]{Marquis}, set
\[
K_0=\{\gamma\in Q^+:\gamma(\alpha_i^\vee)\le0
\text{ for every }i\},
\qquad
K=K_0\cap\Phiimp.
\]
The inequalities defining $K_0$ are equivalent to
$(\gamma\mid\alpha_i)\le0$ for every $i$. The standard description of
positive imaginary roots is
\[
\Phi_{\mathrm{im}}^+
=W\cdot\{\gamma\in K_0\setminus\{0\}:
\operatorname{supp}(\gamma)\text{ is connected}\};
\]
see \cite[\S2.4, Eq.~(2.9)]{Marquis}. In particular, every nonzero
vector in $K_0$ with connected support belongs to $K$.

\begin{lemma}[{\cite[Prop.~5.1]{Kac}}]\label{lem:kac51}
Let $\alpha \in \Phire$ and $\beta \in \Phi$ be non-proportional to $\alpha$. Then there exist nonnegative integers $p,\,q$ such that $\beta+k\alpha \in \Phi$ for  $-p \le k \le q$  with $p-q = \beta(\alpha^\vee)$. In particular, if $(\alpha\mid\beta)>0$ then $p\ge1$, so $\beta-\alpha \in \Phi$.
\end{lemma}

\begin{lemma}[{\cite[Prop.~5.1]{Kac} and \cite[Prop.~5.5]{Kac}}]
\label{lem:multiples}
For $\alpha \in \Phire$ and $n\in\mathbb Z$, $n \alpha \in \Phi$ if and only if $n=\pm 1$.
If $\alpha\in \Phiim$ and $r\in \mathbb{Q}\setminus \{0\}$ such that $r\alpha \in Q$, then $r \alpha \in \Phiim$.  In particular, if $\alpha\in \Phiim$ and $n\in \mathbb{Z}\setminus \{0\}$, then $n\alpha \in \Phiim$.
\end{lemma}

\begin{lemma}[{\cite[Thm.~4.4]{Marquis}}]\label{lem:marquis}
Let $\alpha, \beta\in \Phi$. If $(\alpha \mid \beta)<0$ then $[x,y]\neq 0$ for every nonzero $x\in \mathfrak{g}_\alpha$  and $y\in \mathfrak{g}_\beta$ such that $\mathbb{C}x\neq \mathbb{C}y$.
In particular, if $\alpha,\beta \in \Phi$ satisfy $\alpha\ne\beta$, $\alpha+\beta\ne0$, and $(\alpha\mid\beta)<0$, then $\alpha+\beta \in \Phi$.
\end{lemma}

\begin{lemma}[{\cite[Lem.~3.6(1),(3)]{Marquis}}]\label{lem:negpairing}
For $\alpha,\beta\in\Phiimp$, $(\alpha\mid\beta)\le0$.
If $(\alpha\mid\beta)<0$, then $\alpha+\beta\in\Phiimp$.
\end{lemma}

\begin{lemma}[{\cite[Lem.~3.2]{Marquis}}]\label{lem:fundamental-representative}
For every $\mu\in\Phiimp$, there exists $v\in W$ such that
$v\mu\in K$ and
\[
\operatorname{supp}(v)\subseteq\operatorname{supp}(\mu),
\qquad
\operatorname{supp}(v\mu)\subseteq\operatorname{supp}(\mu).
\]
Here $\operatorname{supp}(v)$ is the subdiagram induced by the
vertices whose simple reflections occur in a reduced expression
for $v$; it is independent of the chosen reduced expression.
\end{lemma}

\begin{lemma}[Consequence of
{\cite[Lem.~3.4]{Marquis}}]\label{lem:separation}
Let $\gamma_1,\dots,\gamma_r\in\Phiimp$ satisfy
$\gamma_i+\gamma_j\notin\Phi$ whenever $i\ne j$.
There exists $w\in W$ such that $w\gamma_i\in K$ for every $i$, and
$\operatorname{supp}(w\gamma_1),\dots,\operatorname{supp}(w\gamma_r)$
are distinct connected components of their union.
In particular, if $r\ge2$, then $\gamma_1+\cdots+\gamma_r\notin\Phi$.
\end{lemma}

\begin{proof}
The simultaneous separation statement is \cite[Lem.~3.4]{Marquis}.
For the final assertion, the transformed sum has disconnected support,
so it is not a root. Since $W$ preserves $\Phi$, the original sum is
not a root either.
\end{proof}

The following addition lemma follows by adapting the argument
in the proof of \cite[Lem.~3.6(2)]{Marquis}. We include the proof
since the lemma is used in the proof of the main theorem.

\begin{lemma}\label{lem:connected-addition}
Let $\lambda\in K$ and $\mu\in\Phiimp$.
If $\operatorname{supp}(\lambda)\cup\operatorname{supp}(\mu)$
is connected, then $\lambda+\mu\in\Phiimp$.
\end{lemma}

\begin{proof}
Since $\lambda\in K_0$ and $\mu\in Q^+$,
$(\lambda\mid\mu)\le0$. If this pairing is negative,
Lemma~\ref{lem:negpairing} gives the result.
Suppose therefore that $(\lambda\mid\mu)=0$. Writing
$\mu=\sum_i m_i\alpha_i$, each term
$m_i(\lambda\mid\alpha_i)$ is nonpositive. Consequently,
\[
(\lambda\mid\alpha_i)=0
\qquad\text{for every }\alpha_i\in\operatorname{supp}(\mu).
\]
We claim that
$\operatorname{supp}(\mu)\subseteq\operatorname{supp}(\lambda)$.
Otherwise, connectedness of the union gives a vertex
$\alpha_i\in\operatorname{supp}(\mu)\setminus
\operatorname{supp}(\lambda)$ adjacent to a vertex in
$\operatorname{supp}(\lambda)$. Since $\alpha_i\notin\operatorname{supp}(\lambda)$,
$(\lambda\mid\alpha_i)$ is a sum of nonpositive terms.
At least one term is strictly negative, because $\alpha_i$ is
adjacent to a vertex in $\operatorname{supp}(\lambda)$.
Thus $(\lambda\mid\alpha_i)<0$, a contradiction.

By Lemma~\ref{lem:fundamental-representative}, choose $v\in W$
with $v\mu\in K$,
$\operatorname{supp}(v)\subseteq\operatorname{supp}(\mu)$, and
$\operatorname{supp}(v\mu)\subseteq\operatorname{supp}(\mu)$.
Every simple reflection in a reduced expression for $v$ therefore
corresponds to a vertex in $\operatorname{supp}(\mu)$ and fixes
$\lambda$, by the orthogonality just established.
Hence $v\lambda=\lambda$.
Both $\lambda$ and $v\mu$ belong to $K_0$, and
$\operatorname{supp}(v\mu)\subseteq\operatorname{supp}(\lambda)$.
Thus $\lambda+v\mu\in K_0$ has connected support
$\operatorname{supp}(\lambda)$, so it is a positive imaginary root.
Applying $v^{-1}$ gives $\lambda+\mu\in\Phiimp$.
\end{proof}

\section{Finite-sum closure}\label{main}

In this section, we extend the finite-sum closure property from semisimple and affine Kac--Moody
algebras \cite{DouglasRepka,DouglasAli} to all symmetrizable Kac--Moody
algebras.

\begin{theorem}[Finite-sum closure]\label{thm:main}
Let $\g$ be a symmetrizable Kac--Moody algebra. Then its root system
$\Phi$ has the finite-sum closure property.
\end{theorem}

\begin{proof} Let $S\subseteq\Phi$ be closed, and let $\beta_1,\dots,\beta_n\in S$ have sum $\eta\in\Phi$. We proceed by strong induction on $n$. The case $n=1$ is trivial and $n=2$ is the definition of closedness.

\vspace{0.5em}

If $\beta_i=-\beta_j$ for some $i\ne j$, removing this pair
leaves $n-2$ elements of $S$ summing to $\eta$, so induction
gives $\eta\in S$. We may therefore assume that
$\beta_i+\beta_j\ne0$ whenever $i\ne j$, and proceed in cases.

\vspace{0.5em}

\noindent \emph{Case 1: $\eta \in \Phire$.}  Since $(\eta\mid\eta) = \sum_i(\eta\mid\beta_i)>0$, some $(\eta\mid\beta_m)>0$. If $\beta_m$ is proportional to $\eta$, Lemma~\ref{lem:multiples} forces $\beta_m=\eta$, so $\eta=\beta_m\in S$. Otherwise Lemma~\ref{lem:kac51} implies $\eta-\beta_m\in\Phi$. Since $\eta-\beta_m \in \Phi$ is a sum of $n-1$ elements of $S$, the induction hypothesis then gives $\eta-\beta_m\in S$. Therefore, closure implies $(\eta-\beta_m)+\beta_m=\eta\in S$.

\vspace{0.5em}

\noindent \emph{Case 2: $\eta \in \Phiim$, some $\beta_m\in \Phire$.}  If $(\beta_m\mid\beta_k)<0$ for some $k\ne m$, Lemma~\ref{lem:marquis} gives $\beta_m+\beta_k\in\Phi$. Since $\beta_m,\beta_k\in S$ and $S$ is closed, $\beta_m+\beta_k\in S$. Replacing the two summands $\beta_m$ and $\beta_k$ by
$\beta_m+\beta_k$ yields $n-1$ elements of $S$ summing to $\eta$, so $\eta\in S$ by induction.

Otherwise, $(\beta_m\mid\beta_i)\ge 0$ for all $i$. Since $\beta_m$ is real, $(\beta_m\mid\beta_m)>0$, and hence, for each $j\ne m$,
\[
\beta_j(\beta_m^\vee)
=
\frac{2(\beta_j\mid\beta_m)}{(\beta_m\mid\beta_m)}
\ge 0.
\]
Therefore,
\[
\eta(\beta_m^\vee)
=
\beta_m(\beta_m^\vee)+\sum_{j\ne m}\beta_j(\beta_m^\vee)
=
2+\sum_{j\ne m}\beta_j(\beta_m^\vee)
\ge 2.
\]
Since $\eta$ is imaginary, it is non-proportional to the real $\beta_m$ (Lemma~\ref{lem:multiples}), so Lemma~\ref{lem:kac51} applies, yielding  $\eta-\beta_m \in \Phi$. Since $\eta-\beta_m\in \Phi$ can be written as the sum of $n-1$ elements of $S$, the induction hypothesis implies $\eta-\beta_m\in S$. By closure $\eta=(\eta-\beta_m)+\beta_m\in S$.

\vspace{0.5em}

\noindent \emph{Case 3: $\eta\in\Phiim$, all $\beta_i\in\Phiim$.}
If $\beta_i+\beta_j\in\Phi$ for some $i\ne j$, closedness and
induction give $\eta\in S$ by replacing this pair with its sum.
Thus we may assume
\begin{equation}\label{eq:no-pair-sums}
\beta_i+\beta_j\notin\Phi\cup\{0\}
\qquad(i\ne j).
\end{equation}
Replacing $S$ and all roots by their negatives if necessary,
assume $\eta\in\Phiimp$. Write
\[
\eta=\lambda_1+\cdots+\lambda_p-\mu_1-\cdots-\mu_q,
\qquad
\lambda_i,\mu_j\in\Phiimp,\qquad p+q=n.
\]
Here $p\ge1$. We may apply Weyl group elements to the entire
configuration, since they preserve closedness and the signs of
imaginary roots.

We first show that $p=1$. Suppose $p\ge2$. By
Lemma~\ref{lem:separation}, we may assume that the supports of the
$\lambda_i$ are distinct connected components of their union $U$.
The equality
\[
\lambda_1+\cdots+\lambda_p
=\eta+\mu_1+\cdots+\mu_q
\]
has only nonnegative coefficients. Hence $\eta$ and every $\mu_j$
are supported in $U$ and, by connectedness of their supports,
each lies in a single component of $U$.

Choose a component $\operatorname{supp}(\lambda_a)$ that does not
contain $\operatorname{supp}(\eta)$. Comparing coefficients there
gives
\[
\lambda_a=\sum_{j\in J}\mu_j,
\qquad
J=\{j:\operatorname{supp}(\mu_j)
\subseteq\operatorname{supp}(\lambda_a)\}.
\]
Since $\lambda_a\ne0$, $J$ is nonempty. If $|J|=1$, this gives
$\lambda_a=\mu_j$, contradicting the absence of opposite
summands. If $|J|\ge2$, Lemma~\ref{lem:separation} says that the
sum on the right is not a root, again a contradiction.
Therefore $p=1$.

Writing $\lambda=\lambda_1$, we now have
\[
\lambda=\eta+\mu_1+\cdots+\mu_q,
\qquad q=n-1\ge2.
\]
By Lemma~\ref{lem:separation}, we may assume that each $\mu_j$
belongs to $K$ and that their supports are mutually disjoint
with no edges between them. Since
\[
\operatorname{supp}(\lambda)
=\operatorname{supp}(\eta)\cup
\bigcup_{j=1}^q\operatorname{supp}(\mu_j)
\]
is connected, each $\operatorname{supp}(\mu_j)$ either intersects
$\operatorname{supp}(\eta)$ or is joined to it by an edge.

Starting with $\eta$, successively add $\mu_2,\dots,\mu_q$.
At every step, the current sum has support containing
$\operatorname{supp}(\eta)$, so its union with the next support
is connected. Applying Lemma~\ref{lem:connected-addition} with the next
$\mu_j\in K$ as the fundamental imaginary root therefore gives
a positive imaginary root at each step. In particular,
\[
\lambda-\mu_1=\eta+\mu_2+\cdots+\mu_q\in\Phiimp.
\]
This contradicts \eqref{eq:no-pair-sums}, since $\lambda$ and
$-\mu_1$ are two summands in the transformed configuration.
Thus some pair must have a root sum, and induction completes
the proof.
\end{proof}

\section{Application to wide regular subalgebras}\label{sec:wide}

We recall the setup of \cite{DouglasAli}. A subalgebra $\mathfrak{s}\subseteq\mathfrak{g}$ is 
\emph{regular} if it is normalized by some Cartan subalgebra
of $\mathfrak{g}$. Fix such a Cartan subalgebra $\mathfrak{h}$.
Then $\mathfrak{s}$ admits a root space decomposition
\[
\mathfrak{s}
=
(\mathfrak{s}\cap \mathfrak{h})
\oplus
\bigoplus_{\alpha\in \Phi} (\mathfrak{s}\cap \mathfrak{g}_\alpha),
\]
and we associate to $\mathfrak{s}$ the subset
\[
T=\{\alpha\in\Phi\mid
\mathfrak{s}\cap\mathfrak{g}_\alpha\neq\{0\}\}.
\]

The associated subset $T$ need not be closed
\cite[Example~2.1]{DouglasAli}. For $S\subseteq\Phi$, let $[S]$ denote its closure; we call $S$ \emph{symmetric} if $S=-S$.
We also write
\[
\operatorname{cl}_{\mathrm{re}}(S)
\]
for the smallest real closed subset of $\Phi_{\mathrm{re}}$ containing
$S\cap\Phi_{\mathrm{re}}$, where a subset $R\subseteq\Phi_{\mathrm{re}}$
is \emph{real closed} if $\alpha,\beta\in R$ and
$\alpha+\beta\in\Phi_{\mathrm{re}}$ imply $\alpha+\beta\in R$.

A subalgebra is \emph{wide} if every irreducible integrable highest weight
$\g$-module remains indecomposable upon restriction.  For such a module $V$, there is a weight space decomposition 
\[
V=\bigoplus_{\mu\in \mathfrak{h}^*} V_\mu,~ \text{where}~ V_\mu=\{v\in V  \mid h v=\mu(h)v, ~\text{for all}~ h\in \mathfrak{h}\}.
\]

The \emph{$\mathfrak{s}$-weight graph} of $V$ has vertex set
\[
\{\mu\in\mathfrak h^*\mid V_\mu\neq0\},
\]
with an edge between $\lambda$ and $\mu$ whenever
$\mu-\lambda\in T\cup(-T)$.

For affine $\g$, \cite[Thm.~3.5]{DouglasAli} shows that
\[
\mathfrak{s}\text{ wide}
\quad\Longrightarrow\quad
[T\cup(-T)]=\Phi.
\]
This same property also holds for semisimple Lie algebras
\cite{DouglasRepka}. Theorem~\ref{thm:main} allows the same argument to
be extended to every symmetrizable Kac--Moody algebra.

\begin{theorem}[Necessary condition for wideness]
\label{thm:wide-necessary}
Let $\g$ be a symmetrizable Kac--Moody algebra, and let
$\mathfrak{s}\subseteq\g$ be a regular subalgebra with associated
subset $T\subseteq\Phi$. Then
\[
\mathfrak{s}\text{ wide}
\quad\Longrightarrow\quad
[T\cup(-T)]=\Phi.
\]
\end{theorem}

\begin{proof}
Set $S:=[T\cup(-T)]$. As in \cite[Thm.~3.5]{DouglasAli}, $S$ is
symmetric and closed. Suppose $S\ne\Phi$. Then $\Pi\nsubseteq S$:
indeed, if $\Pi\subseteq S$, every positive root is a finite sum of
simple roots, so Theorem~\ref{thm:main} gives $\Phi^+\subseteq S$,
and symmetry gives $\Phi^-\subseteq S$, contradicting $S\ne\Phi$.
Choose $\alpha_k\in\Pi\setminus S$.

Let $V=V(\omega_k)$ be the irreducible integrable highest weight
module of highest weight $\omega_k$. Since
$\omega_k(\alpha_k^\vee)=1$, the weight $\omega_k-\alpha_k$ occurs
in $V$. If $\omega_k$ and $\omega_k-\alpha_k$ belonged to the same
connected component of the $\mathfrak{s}$-weight graph, a path
between them would give
\[
-\alpha_k=\beta_1+\cdots+\beta_r,
\qquad
\beta_j\in T\cup(-T)\subseteq S.
\]
Theorem~\ref{thm:main} would then give $-\alpha_k\in S$, and symmetry
would give $\alpha_k\in S$, a contradiction. Thus the weight graph
is disconnected. As in \cite[Thm.~3.5]{DouglasAli}, its connected
components yield a nontrivial decomposition of $V$ into
$\mathfrak{s}$-submodules, contradicting wideness. Hence
$[T\cup(-T)]=\Phi$.
\end{proof}

\begin{remark}[Wideness criteria]\label{rem:wideness-both}
If $\mathfrak{h}\subseteq\mathfrak{s}$, then
\cite[Thm.~3.3]{DouglasAli} and Theorem~\ref{thm:wide-necessary} give
\[
\Pi\subseteq\operatorname{cl}_{\mathrm{re}}(T\cup(-T))
\quad\Longrightarrow\quad
\mathfrak{s}\text{ is wide}
\quad\Longrightarrow\quad
[T\cup(-T)]=\Phi.
\]
Thus the real-closure condition is sufficient for wideness,
while $[T\cup(-T)]=\Phi$ is necessary.
\end{remark}

\begin{remark}[A remaining question]
If $\mathfrak{h}\subseteq\mathfrak{s}$ and
$[T\cup(-T)]=\Phi$, must $\mathfrak{s}$ be wide?
An affirmative answer would make the necessary condition
in Theorem~\ref{thm:wide-necessary} sufficient as well.
\end{remark}

\section*{Acknowledgments}

Support for this project was provided by PSC-CUNY Award
\#GR-00017525. The author used AI to assist with proofreading
and editing. 
The author thanks Timoth\'ee Marquis for answering
a question about a result in his article \cite{Marquis} and Abid Ali
for his thoughtful feedback, which improved the clarity of the article.

\end{document}